\documentclass[11pt]{article}
\usepackage{graphicx,amsmath,amsfonts,amssymb,bbm,amsthm}
\usepackage{psfrag}
\usepackage{epsf}
\usepackage{color}

\makeatletter\@addtoreset {equation}{section}\makeatother

\theoremstyle{plain}
\newtheorem{theo}{Theorem}
\newtheorem{lem}{Lemma}[section]

\theoremstyle{remark}

{\hspace*{\fill}$\rule{.3\baselineskip}{.35\baselineskip}$\end{trivlist}}

\newcommand{\R}{\mathbb{R}}

\renewcommand{\geq}{\geqslant}
\renewcommand{\leq}{\leqslant}

\renewcommand{\phi}{\varphi}
\newcommand{\be}{\begin{eqnarray}}
\newcommand{\ee}{\end{eqnarray}}

\begin{document}

\title{\bf Long-time stability of small FPU solitary waves}

\author{Amjad Khan$^{1}$ and Dmitry Pelinovsky$^{2}$ \\
{\small $^{1}$ Department of Applied Mathematics, Western University, London, ON, Canada, N6A 3K7} \\
{\small $^{2}$ Department of Mathematics, McMaster University, Hamilton, Ontario, Canada, L8S 4K1}  }

\date{\today}
\maketitle

\begin{abstract}
Small-amplitude waves in the Fermi-Pasta-Ulam (FPU) lattice with weakly anharmonic interaction potentials
are described by the generalized Korteweg-de Vries (KdV) equation. Justification of the small-amplitude
approximation is usually performed on the time scale, for which dynamics of the KdV equation is defined.
We show how to extend justification analysis on longer time intervals provided dynamics of the generalized KdV
equation is globally well-posed in Sobolev spaces and either the Sobolev norms are globally bounded
or they grow at most polynomially. The time intervals are extended respectively by the logarithmic or double logarithmic factors
in terms of the small amplitude parameter. Controlling the approximation error on longer time intervals
allows us to deduce nonlinear metastability of small FPU solitary waves from orbital stability of the KdV solitary waves.
\end{abstract}

\section{Introduction}

In this work, we address an open question from \cite{DumaPel} on how to deduce nonlinear metastability or instability
of small Fermi--Pasta--Ulam (FPU) solitary waves from orbital stability or instability of the Korteweg--de Vries (KdV) solitary waves.
Let us consider dynamics of the FPU lattice given by Newton's equations of motion:
\begin{equation}
\label{FPU}
\ddot{x}_n = V'(x_{n+1} - x_n) - V'(x_n - x_{n-1}),
\quad n \in \mathbb{Z},
\end{equation}
where $(x_n)_{n \in \mathbb{Z}}$ is a function of the time
$t\in\mathbb{R}$, with values in $\mathbb{R}^{\mathbb{Z}}$,
the dot denotes the time derivative, and the interaction potential $V$ is smooth.
The coordinate $x_n$ corresponds to the displacement of the $n$-th particle
in a one-dimensional chain from its equilibrium position.
The potential $V$ for anharmonic interactions of particles is taken in the form
\begin{equation}
\label{interaction}
V(u) = \frac{1}{2} u^2 + \frac{\epsilon^2}{p+1} u^{p+1},
\end{equation}
where $p \geq 2$ is integer and the strength of anharmonicity $\epsilon$ can be introduced
by the scaling transformation. The FPU lattice equations (\ref{FPU}) can be rewritten
in the strain variables $u_n := x_{n+1}-x_n$ as follows
\begin{equation}
\label{FPUlattice}
\ddot{u}_n = V'(u_{n+1}) - 2 V'(u_n) + V'(u_{n-1}),
\quad n \in \mathbb{Z}.
\end{equation}
Using the well-known asymptotic multi-scale expansion \cite{BP06,Pego,SW00},
\begin{equation}
\label{ansatz-KdV}
u_n(t) = W(\epsilon(n-t),\epsilon^3 t) + {\rm error \;\; terms},
\end{equation}
yields the generalized KdV equation for the leading-order approximation $W$ given by
\begin{equation}
\label{genKdV}
2 W_{\tau} + \frac{1}{12} W_{\xi \xi \xi} + (W^p)_{\xi} = 0, \quad \xi \in \mathbb{R},
\end{equation}
where $\tau = \epsilon^3 t$ and $\xi = \epsilon (n-t)$.

Local well-posedness of the generalized KdV equation (\ref{genKdV}) in Sobolev spaces $H^s(\mathbb{R})$
is known from the works of Kato \cite{Kato1,Kato2} for $s > \frac{3}{2}$ and Kenig--Ponce--Vega \cite{KPV1,KPV2}
for $s \geq s_p$, where
$$
s_{p=2} = \frac{3}{4}, \quad s_{p=3} = \frac{1}{4}, \quad s_{p = 4} = \frac{1}{12}, \quad
s_{p \geq 5} = \frac{p-5}{2(p-1)}.
$$
For any local solution $W \in C([-\tau_0,\tau_0],H^s(\mathbb{R}))$ of the KdV equation
\eqref{genKdV} with $s \geq 6$ and $\tau_0 > 0$, the error terms in the asymptotic multi-scale expansion (\ref{ansatz-KdV}) can be controlled
as follows. There exist positive constants $\epsilon_0$ and $C_0$ such that,
for all $\epsilon \in (0,\epsilon_0)$, when initial data
$(u_{\rm in},\dot{u}_{\rm in}) \in \ell^2(\mathbb{R})$ satisfy
\begin{equation}
\label{bound-initial-time}
\| u_{\rm in} - W(\epsilon \cdot,0) \|_{\ell^2} +
\| \dot{u}_{\rm in} + \epsilon \partial_{\xi} W(\epsilon \cdot,0) \|_{\ell^2} \leq \epsilon^{3/2},
\end{equation}
the unique solution $(u,\dot{u})$ to the FPU equation (\ref{FPUlattice}) with initial data $(u_{\rm in},\dot{u}_{\rm in})$
belongs to $C^1([-\tau_0\epsilon^{-3},\tau_0 \epsilon^{-3}],\ell^2(\mathbb{Z}))$ and satisfies
\begin{equation}
\label{bound-final-time}
\| u(t) - W(\epsilon (\cdot - t), \epsilon^3 t) \|_{\ell^2} +
\| \dot{u}(t) + \epsilon \partial_{\xi} W(\epsilon (\cdot - t), \epsilon^3 t)  \|_{\ell^2}
\leq C_0 \epsilon^{3/2}, \quad t \in \left[-\tau_0 \epsilon^{-3},\tau_0 \epsilon^{-3}\right].
\end{equation}
The proof of this result is based on the energy estimates and Gronwall's inequality \cite{BP06,Pego,SW00}.

Bound (\ref{bound-final-time}) suggests that small-amplitude FPU solitary waves
are metastable or unstable if the KdV solitary waves are orbitally stable or unstable.
Indeed, the generalized KdV equation (\ref{genKdV}) is known to have orbitally stable solitary waves
for $p = 2,3,4$ and orbitally unstable solitary waves for $p \geq 5$ \cite{PW1}.
However, this simple and widely accepted analogy appears in apparent contradiction
with the energy arguments found in \cite{DumaPel} suggesting unconditional metastability
of small-amplitude FPU solitary waves on the time scale of $\mathcal{O}(\epsilon^{-3})$.

The metastability result from \cite{DumaPel} can be formulated as follows.
Let us denote the traveling-wave solutions of the FPU equation (\ref{FPUlattice})
by $u_{\rm trav}$. Then, for every $\tau_0 > 0$,
there exist positive constants $\epsilon_0$, $\delta_0$ and $C_0$ such that,
for all $\epsilon \in (0,\epsilon_0)$, when initial data $(u_{\rm in},\dot{u}_{\rm in}) \in \ell^2(\mathbb{R})$ satisfy
\begin{equation}
\label{bound-initial}
\delta := \| u_{\rm in} - u_{\rm trav}(0) \|_{\ell^2} +
\| \dot{u}_{\rm in} - \dot{u}_{\rm trav}(0) \|_{\ell^2} \leq \delta_0,
\end{equation}
the unique solution $(u,\dot{u})$ to the FPU equation (\ref{FPUlattice}) with initial data $(u_{\rm in},\dot{u}_{\rm in})$
belongs to $C^1([-\tau_0\epsilon^{-3},\tau_0\epsilon^{-3}],\ell^2(\mathbb{Z}))$ and satisfies
\begin{equation}
\label{bound-final}
\| u(t) - u_{\rm trav}(t) \|_{l^2} +
\| \dot{u}(t) - \dot{u}_{\rm trav}(t) \|_{l^2}
\leq C_0 \delta, \quad t \in \left[-\tau_0 \epsilon^{-3},\tau_0 \epsilon^{-3}\right].
\end{equation}
Similarly to the bound (\ref{bound-final-time}), the bound
(\ref{bound-final}) is also proved with the energy estimates and Gronwall's inequality
complemented with the asymptotic scaling of small-amplitude FPU traveling waves $u_{\rm trav}$
near the KdV solitary waves \cite{DumaPel}.
Nevertheless, the initial data $(u_{\rm in},\dot{u}_{\rm in})$  can be modulated on any spatial scale.

Bound (\ref{bound-final}) suggests unconditional metastability of small-amplitude
FPU solitary waves up to the time scale of $\mathcal{O}(\epsilon^{-3})$ for every $p \geq 2$.
This may be viewed as a contradiction with the bound (\ref{bound-final-time})
that suggests instability of small-amplitude FPU solitary waves at the time
scale of $\mathcal{O}(\epsilon^{-3})$ for $p \geq 5$ because
the corresponding KdV solitary waves are unstable for $p \geq 5$ \cite{PW1}.

Of course, no contradiction arises because the energy methods used in the proof of
the upper bounds on the approximation errors (\ref{bound-final-time}) and (\ref{bound-final})
yield constants $C_0$ that grow exponentially in time $\tau_0$, that is, on the time scale of $\mathcal{O}(\epsilon^{-3})$.
As a result, the exponential divergence of the constant $C_0$ cannot be
distinguished from the exponential instability of the KdV solitary waves in the case $p \geq 5$.
However, this observation also shows that the bound (\ref{bound-final}) is not a reliable evidence to conclude on
metastability of the small-amplitude FPU solitary waves in the case of $p = 2,3,4$.

Nonlinear stability of small-amplitude FPU solitary waves in the case of the classical KdV equation
with $p = 2$ was studied in the series of papers by Friesecke \& Pego \cite{Pego} based on the orbital
and asymptotic stability of the KdV solitons \cite{PW2}.  Similarly, asymptotic stability of several solitary
waves was studied by Mizumachi \cite{miz1,miz2} and Benes, Hoffmann \& Wayne \cite{HW1,HW2} also in the case $p = 2$.
Derivation and analysis of small-amplitude FPU solitary waves were recently
generalized for polyatomic FPU lattices in \cite{GMWZ}.

In the present work, we extend the bound (\ref{bound-final-time}) on the approximation error to
longer time intervals provided dynamics of the generalized KdV
equation (\ref{genKdV}) is globally well-posed in Sobolev spaces and either the Sobolev norms
are globally bounded or they grow at most polynomially.

For the integrable cases $p = 2$ and $p = 3$, a uniform bound on the $H^s(\mathbb{R})$ norms
for any $s \in \mathbb{N}$ can be obtained from conserved quantities of the KdV and modified
KdV hierarchies \cite{Nguenn,Gardner}. For the non-integrable case $p = 4$, the global
solution is controlled in $H^1(\mathbb{R})$ by using the energy conservation, while
the $H^s(\mathbb{R})$ norms with $s \geq 2$ can grow at most polynomially.
In particular, it was proved by Staffilani \cite{Staf}
that for any $s \geq 2$, there exists a constant $C_s$ such that
the unique solution of the generalized KdV equation (\ref{genKdV}) with $p = 4$ satisfies
\begin{equation}
\label{growth-norm}
\| W(\tau) \|_{H^s} \leq C_s |\tau|^{s-1} \quad \mbox{\rm as} \quad |\tau| \to \infty.
\end{equation}
Global solutions to the generalized KdV equation (\ref{genKdV}) with $p \geq 5$ exist
and scatter to zero \cite{KPV2} if the $H^{s_p}(\mathbb{R})$ norm of initial data is small, where
\begin{equation}
\label{s-p}
s_p = \frac{p-5}{2(p-1)}, \quad p \geq 5.
\end{equation}
For the global solutions scattering to zero in the case $p \geq 5$,
the $H^1(\mathbb{R})$ norm is again controlled by the energy conservation \cite{KPV2},
and the polynomial bound (\ref{growth-norm}) holds \cite{Staf}.

The following theorems
extend the bound (\ref{bound-final-time}) on the approximation error to longer time intervals.
The two cases have to be considered separately, depending whether
the $H^s(\mathbb{R})$ norm of the KdV solution is globally bounded
or may grow at most polynomially.

\begin{theo}
\label{theorem-justification-1}
Let $W \in C(\mathbb{R},H^s(\mathbb{R}))$ be a global solution to the generalized KdV equation
\eqref{genKdV} with either $p = 2$ or $p = 3$ for some integer $s \geq 6$.
For fixed $r \in \left( 0, \frac{1}{2} \right)$, there exist positive constants $\epsilon_0$, $C$, and $K$ such that,
for all $\epsilon \in (0,\epsilon_0)$, when initial data $(u_{\rm in},\dot{u}_{\rm in}) \in \ell^2(\mathbb{R})$ satisfy
\begin{equation}
\label{bound-initial-time-1}
\| u_{\rm in} - W(\epsilon \cdot,0) \|_{\ell^2} +
\| \dot{u}_{\rm in} + \epsilon \partial_{\xi} W(\epsilon \cdot,0) \|_{\ell^2} \leq \epsilon^{3/2},
\end{equation}
the unique solution $(u,\dot{u})$ to the FPU equation (\ref{FPUlattice}) with initial data $(u_{\rm in},\dot{u}_{\rm in})$
belongs to
$$
C^1([-t_0(\epsilon),t_0(\epsilon)],\ell^2(\mathbb{Z}))
$$
with $t_0(\epsilon) := r K^{-1} \epsilon^{-3} |\log(\epsilon)|$ and satisfies
\begin{equation}
\label{bound-final-time-1}
\| u(t) - W(\epsilon (\cdot - t), \epsilon^3 t) \|_{\ell^2} +
\| \dot{u}(t) + \epsilon \partial_{\xi} W(\epsilon (\cdot - t), \epsilon^3 t)  \|_{\ell^2}
\leq C \epsilon^{3/2-r}, \quad t \in \left[-t_0(\epsilon),t_0(\epsilon)\right].
\end{equation}
\end{theo}

\begin{theo}
\label{theorem-justification-2}
Let $W \in C(\mathbb{R},H^s(\mathbb{R}))$ be a global solution to the generalized KdV equation
\eqref{genKdV} with either $p = 4$ or $p \geq 5$ and small $\| W(0) \|_{H^{s_p}}$, for some integer $s \geq 6$.
For fixed $r \in \left( 0, \frac{1}{2} \right)$, there exist positive constants $\epsilon_0$, $C$, and $K$
such that, for all $\epsilon \in (0,\epsilon_0)$, when initial data
$(u_{\rm in},\dot{u}_{\rm in}) \in \ell^2(\mathbb{R})$ satisfy
\begin{equation}
\label{bound-initial-time-2}
\| u_{\rm in} - W(\epsilon \cdot,0) \|_{\ell^2} +
\| \dot{u}_{\rm in} + \epsilon \partial_{\xi} W(\epsilon \cdot,0) \|_{\ell^2} \leq \epsilon^{3/2},
\end{equation}
the unique solution $(u,\dot{u})$ to the FPU equation (\ref{FPUlattice}) with initial data $(u_{\rm in},\dot{u}_{\rm in})$
belongs to
$$
C^1([-t_0(\epsilon),t_0(\epsilon)],\ell^2(\mathbb{Z}))
$$
with $t_0(\epsilon) := (2 p K)^{-1} \epsilon^{-3} \log\left( r |\log(\epsilon)| \right)$ and satisfies
\begin{equation}
\label{bound-final-time-2}
\| u(t) - W(\epsilon (\cdot - t), \epsilon^3 t) \|_{\ell^2} +
\| \dot{u}(t) + \epsilon \partial_{\xi} W(\epsilon (\cdot - t), \epsilon^3 t)  \|_{\ell^2}
\leq C \epsilon^{3/2-r}, \quad t \in \left[-t_0(\epsilon),t_0(\epsilon)\right].
\end{equation}
\end{theo}

We note that the final time of the dynamics of the generalized KdV equation
(\ref{genKdV}) given by $\tau_0(\epsilon) := \epsilon^3 t_0(\epsilon)$ depends on $\epsilon$
and satisfies $\tau_0(\epsilon) \to \infty$ as $\epsilon \to 0$ both
in Theorems \ref{theorem-justification-1} and \ref{theorem-justification-2}.
Bounds (\ref{bound-final-time-1}) and (\ref{bound-final-time-2})
allow us to deduce nonlinear metastability of small FPU solitary waves in the solution
$(u,\dot{u})$ from orbital stability of the KdV solitary waves in
the solution $W$ to the generalized KdV equation (\ref{genKdV}).
In particular, solitary waves of the generalized KdV equation (\ref{genKdV})
are orbitally stable for $p = 2,3,4$, and so are small-amplitude FPU solitary waves
on long but finite time intervals.

Solitary waves of the generalized KdV equation (\ref{genKdV}) are unstable for $p \geq 5$ and the class
of global solutions considered in Theorem \ref{theorem-justification-2} for $p \geq 5$
excludes solitary waves. On the other hand,
dynamics of the small-amplitude waves in the FPU lattice resembles scattering dynamics
of small solutions to the generalized KdV equation (\ref{genKdV}) with $p \geq 5$ \cite{KPV2}.

Extended approximations on longer time intervals become increasingly popular in
the justification analysis of amplitude equations in various evolutionary problems.
One of the pioneer works is developed by Lannes \& Rauch in the context of
validity of the nonlinear geometric optics equations \cite{Lannes}.
Extended time intervals modified by a logarithmic factor of $\epsilon$ were introduced
in the justification of the discrete nonlinear Schr\"{o}dinger equation
in the context of the FPU lattices \cite{DumaJames} and the Klein--Gordon lattices \cite{PPP}.
Our work addresses the extended time intervals in the justification 
of the KdV equation in the context of the FPU lattices.

The rest of the paper is structured as follows. Section 2 represents the basic set up
for justification analysis of the generalized KdV equation (\ref{genKdV}) from
the FPU lattice equation (\ref{FPUlattice}). Justification arguments on the KdV time scale are
well-known and follow the formalism described in \cite{SW00} with a refinement given in \cite{DumaPel}. Sections 3 and 4
present details of the proofs of Theorems \ref{theorem-justification-1} and \ref{theorem-justification-2}.
This part is original and represents the main result of this paper.

\vspace{0.25cm}

{\bf Acknowledgement.} The work of A. Khan was performed during MSc program at McMaster University in 2013-2015.
The work of D. Pelinovsky is supported by the NSERC grant. The authors thank E. Dumas and  G. Schneider
for discussions and collaborations.

\section{Justification setup}

The scalar second-order equation (\ref{FPUlattice}) can be rewritten as the following first-order evolution system
\begin{equation}
\label{FPUvector}
\left\{ \begin{array}{l} \dot{u}_n = q_{n+1} - q_n, \\
\dot{q}_n = u_n - u_{n-1} + \epsilon^2 (u_n^p - u_{n-1}^p), \end{array} \right. \quad n \in \mathbb{Z}.
\end{equation}
Local solutions  $(u,q) \in C^1([-t_0,t_0],\ell^2(\mathbb{Z}))$ exists by standard Picard iterations,
thanks to analyticity of the power nonlinearity with $p \in \mathbb{N}$ and to the boundness
of the shift operators on $\ell^2(\mathbb{Z})$. For a given initial data $(u_{\rm in},q_{\rm in}) \in \ell^2(\mathbb{Z})$,
local solutions are extended to the global solutions $(u,q) \in C^1(\mathbb{R},\ell^2(\mathbb{Z}))$ by decreasing the values of $\epsilon$
thanks to the energy conservation
\begin{equation}
\label{energy}
H := \frac{1}{2} \sum_{n \in \mathbb{Z}} \left( q_n^2 + u_n^2 + \frac{2 \epsilon^2}{p+1} u_n^{p+1} \right).
\end{equation}
If $p$ is odd, no constraints on $\epsilon$ arise for existence of global solutions $(u,q) \in C^1(\mathbb{R},\ell^2(\mathbb{Z}))$
to the FPU lattice equations (\ref{FPUvector}).

Let us use the decomposition
\begin{equation}
\label{decomposition-time}
u_n(t) = W(\epsilon (n-t), \epsilon^3 t) + \mathcal{U}_n(t), \quad
q_n(t) = P_{\epsilon}(\epsilon (n-t),\epsilon^3 t) + \mathcal{Q}_n(t), \quad n \in \mathbb{Z},
\end{equation}
where $W(\xi,\tau)$ is a suitable solution to the generalized KdV equation (\ref{genKdV})
(and thus $W$ is $\epsilon$-independent), whereas the $\epsilon$-dependent function $P_{\epsilon}(\xi,\tau)$ is found from the truncation
of the first equation of the system (\ref{FPUvector}) rewritten as
\begin{equation} \label{FPU-lattice-first-trucated}
P_{\epsilon}(\xi + \epsilon,\tau) - P_{\epsilon}(\xi,\tau) =
-\epsilon \partial_{\xi} W(\xi,\tau) + \epsilon^3 \partial_{\tau} W(\xi,\tau).
\end{equation}
Looking for an approximate solution $P_{\epsilon}$ to this equation up to and including
the formal order of $\mathcal{O}(\epsilon^3)$, we write
\begin{equation}
\label{expansion-P}
P_{\epsilon} := P^{(0)} + \epsilon P^{(1)} + \epsilon^2 P^{(2)} + \epsilon^3 P^{(3)}
\end{equation}
and collect the corresponding powers of $\epsilon$. After routine computations
(see, e.g., \cite{DumaPel}), we obtain
\begin{equation}
\label{expansion-P-explicit}
P_{\epsilon} := - W + \frac{1}{2} \epsilon \partial_{\xi} W - \frac{1}{8} \epsilon^2 \partial_{\xi}^2 W
- \frac{1}{2} \epsilon^2 W^p + \frac{1}{48} \epsilon^3 \partial_{\xi}^3 W + \frac{1}{4} \epsilon^3 p W^{p-1} \partial_{\xi} W.
\end{equation}
Note that equation (\ref{FPU-lattice-first-trucated})
is satisfied by the expansion (\ref{expansion-P-explicit}) only approximately, up to the terms
of the formal order $\mathcal{O}(\epsilon^5)$.

Functions $W$ and $P_{\epsilon}$ depend on $\xi = \epsilon (n-t)$. In order to be able to control
the residual terms of the KdV approximation, we will use the following lemma proved
in \cite{DumaPel} (see also \cite{SW00} for a weaker result).

\begin{lem} \label{Hs-to-l2}
There exists $C>0$ such that for all $X \in H^1(\R)$ and $\epsilon \in (0,1)$,
\begin{equation*}
\| x \|_{\ell^2} \leq C \epsilon^{-1/2} \| X \|_{H^1},
\end{equation*}
where $x_n := X(\epsilon n)$, $n \in \mathbb{Z}$.
\end{lem}

Substituting the decompositions (\ref{decomposition-time}) and (\ref{expansion-P-explicit})
into the FPU lattice equations (\ref{FPUvector}),
we obtain the evolution problem for the error terms
\begin{eqnarray}
\left\{
\begin{split}
& \dot{\mathcal{U}}_n = \mathcal{Q}_{n+1} - \mathcal{Q}_n + {\rm Res}_n^{(1)}(t), \\
& \dot{\mathcal{Q}}_n = \mathcal{U}_n - \mathcal{U}_{n-1} + \mathcal{R}_n(W,\mathcal{U}) + {\rm Res}_n^{(2)}(t) \\
& \qquad \quad + p \epsilon^2 W^{p-1}(\epsilon(n-t),\epsilon^3t) \mathcal{U}_{n}
- p \epsilon^2 W^{p-1}(\epsilon(n-1-t),\epsilon^3t) \mathcal{U}_{n-1},
\end{split} \right.
\label{FPU-lattice-time}
\end{eqnarray}
where the residual and nonlinear terms are given by
{\small \begin{eqnarray*}
{\rm Res}_n^{(1)}(t) & := & \epsilon \partial_{\xi} W(\epsilon(n-t),\epsilon^3t)
- \epsilon^3 \partial_{\tau} W(\epsilon(n-t),\epsilon^3t) + P_\epsilon(\epsilon(n+1-t),\epsilon^3t) - P_\epsilon(\epsilon(n-t),\epsilon^3t), \\
{\rm Res}_n^{(2)}(t) & := & \epsilon \partial_{\xi} P_{\epsilon}(\epsilon(n-t),\epsilon^3t)
- \epsilon^3 \partial_{\tau} P_{\epsilon}(\epsilon(n-t),\epsilon^3t) + W(\epsilon(n-t),\epsilon^3t) - W(\epsilon(n-1-t),\epsilon^3t) \\
& \phantom{tex} & + \epsilon^2 \left[ W^p(\epsilon(n-t),\epsilon^3t) - W^p(\epsilon(n-1-t),\epsilon^3t) \right]
\end{eqnarray*}}and
\begin{equation*}
\mathcal{R}_n(W,\mathcal{U})(t) :=
\epsilon^2 \sum_{k=2}^p \left( \begin{array}{c} p \\ k \end{array} \right)
\left[ W^{p-k}(\epsilon(n-t),\epsilon^3 t) \mathcal{U}_n^k - W^{p-k}(\epsilon(n-1-t),\epsilon^3 t) \mathcal{U}_{n-1}^k \right].
\end{equation*}
The following result provide bounds on the residual and nonlinear terms, provided that the function $W(\epsilon (\cdot - t), \epsilon^3 t)$
belong to $H^s(\mathbb{R})$ for $s \geq 6$ and satisfies the generalized KdV equation (\ref{genKdV}).

\begin{lem} \label{lemma-residual}
Let $W \in C([-\tau_0,\tau_0],H^s(\mathbb{R}))$ be a solution to the generalized KdV equation
\eqref{genKdV}, for an integer $s \geq 6$ and $\tau_0 > 0$. Define
\begin{equation}
\label{delta}
\delta := \sup_{\tau \in [-\tau_0,\tau_0]} \| W(\tau) \|_{H^s}.
\end{equation}
There exists a positive $\delta$-independent constant $C$ such that
the residual and nonlinear terms satisfy
\begin{equation}
\label{bound-rem}
\| {\rm Res}^{(1)}(t) \|_{\ell^2} + \| {\rm Res}^{(2)}(t) \|_{\ell^2} \leq C \left( \delta + \delta^{2p-1} \right) \epsilon^{9/2}
\end{equation}
and
\begin{equation}
\label{bound-non}
\| \mathcal{R}(W,\mathcal{U})(t) \|_{\ell^2} \leq \epsilon^2 C \left(
\delta^{p-2} + \| \mathcal{U} \|_{\ell^2}^{p-2} \right) \| \mathcal{U} \|_{\ell^2}^2
\end{equation}
for every $t\in[-\tau_0\epsilon^{-3},\tau_0\epsilon^{-3}]$ and $\epsilon \in (0,1)$.
\end{lem}

\begin{proof}
By constructing $P_{\epsilon}$ in  \eqref{expansion-P-explicit},
we have canceled all terms in ${\rm Res}^{(1)}(t)$ up to and including the formal order $\epsilon^4$.
The remainder terms can be written in the closed form with Taylor's theorem as follows:
\begin{equation}
\label{remainder-terms}
\epsilon^5 \int_0^1 (1-r)^4 \partial_\xi^5 W(\epsilon(n-t+r),\epsilon^3t) {\rm d}r
\quad \mbox{and} \quad
\epsilon^5 \int_0^1 (1-r)^2 \partial_\xi^3 W^p(\epsilon(n-t+r),\epsilon^3t) {\rm d}r.
\end{equation}
The associated $\ell^2(\mathbb{Z})$ norm is estimated by $\left( \| W \|_{H^6} + \| W \|_{H^6}^{p} \right) \epsilon^{9/2}$,
thanks to Lemma~\ref{Hs-to-l2}.

By formal expansion of ${\rm Res}^{(2)}(t)$ in $\epsilon$, we confirm that
all terms up to and including the formal order $\epsilon^4$ are canceled
if $W$ satisfies the generalized KdV equation (\ref{genKdV}). On the other hand,
the remainder terms contains terms like (\ref{remainder-terms}) and additionally terms like
\begin{equation}
\epsilon^5 \int_0^1 \partial_\xi W^{2p-1}(\epsilon(n-t+r),\epsilon^3t) {\rm d}r.
\end{equation}
The associated $\ell^2(\mathbb{Z})$ norm is estimated by $\left( \| W \|_{H^6} + \| W \|_{H^6}^{p} + \| W \|_{H^6}^{2p-1}\right) \epsilon^{9/2}$,
thanks to Lemma~\ref{Hs-to-l2}. Hence, we obtain the bound \eqref{bound-rem} by
interpolating between the end point terms.

To prove the bound (\ref{bound-non}), we interpolate the binomial expansion for
$\mathcal{R}_n(W,\mathcal{U})(t)$ between the end point terms and obtain for some $C > 0$:
\begin{eqnarray*}
\| \mathcal{R}(W,\mathcal{U})(t) \|_{\ell^2}
\leq C \epsilon^2 \left( \| W(\epsilon(\cdot - t),\epsilon^3 t)) \|^{p-2}_{L^{\infty}} \| \mathcal{U} \|_{\ell^2}^2
+  \| \mathcal{U} \|_{\ell^p}^p \right).
\end{eqnarray*}
By using continuous embeddings of $H^6(\mathbb{R})$ into $L^{\infty}(\mathbb{R})$ and $\ell^p(\mathbb{Z})$ to $\ell^2(\mathbb{Z})$, we obtain
the bound (\ref{bound-non}).
\end{proof}

For a local solution $(\mathcal{U},\mathcal{Q}) \in C^1([-t_0,t_0],\ell^2(\mathbb{Z}))$ with some $t_0 > 0$
to the perturbed FPU lattice equations (\ref{FPU-lattice-time}), we define the energy-type quantity
\begin{equation}
\mathcal{E}(t) := \frac{1}{2} \sum_{n \in \mathbb{Z}} \left[
\mathcal{Q}_n^2(t) + \mathcal{U}_n^2(t) + \epsilon^2 p W^{p-1}(\epsilon(n-t),\epsilon^3t)) \mathcal{U}^2_{n}(t)
\right].
\label{energy-type}
\end{equation}
The following lemma describes properties of the energy-type quantity $\mathcal{E}(t)$.

\begin{lem}
\label{lemma-energy}
Let $W \in C([-\tau_0,\tau_0],H^s(\mathbb{R}))$ be a solution to the generalized KdV equation
\eqref{genKdV}, for an integer $s \geq 6$ and $\tau_0 > 0$. Let $\epsilon_0 > 0$ be defined by
\begin{equation}
\label{epsilon-0}
\epsilon_0 := \min\left\{ 1, (2p)^{-1/2} \left( \sup_{\tau \in [-\tau_0,\tau_0]}\| W(\cdot,\tau) \|_{L^{\infty}} \right)^{-(p-1)/2} \right\}
\end{equation}
For every $\epsilon \in (0,\epsilon_0)$ and for every local solution
$(\mathcal{U},\mathcal{Q}) \in C^1([-\tau_0 \epsilon^{-3},\tau_0 \epsilon^{-3}],\ell^2(\mathbb{Z}))$
to system (\ref{FPU-lattice-time}), the energy-type quantity (\ref{energy-type}) is coercive with the bound
\begin{equation}
\label{coercivity}
\| \mathcal{Q}(t) \|_{l^2}^2 + \| \mathcal{U}(t) \|_{\ell^2}^2 \leq 4 \mathcal{E}(t),
\quad t \in (-\tau_0 \epsilon^{-3},\tau_0 \epsilon^{-3}).
\end{equation}
Moreover, when $\delta$ is defined by (\ref{delta}),
there exists a positive ($\epsilon,\delta$)-independent constant $C$ such that
\begin{equation}
\label{derivative-estimates}
\left| \frac{d \mathcal{E}}{d t} \right| \leq
C \mathcal{E}^{1/2} \left[ (\delta + \delta^{2p-1}) \epsilon^{9/2} + \epsilon^3 (\delta^{p-1} + \delta^{2p-2})
\mathcal{E}^{1/2} + \epsilon^2 (\delta^{p-2} + \mathcal{E}^{(p-2)/2}) \mathcal{E} \right],
\end{equation}
for every $t\in[-\tau_0\epsilon^{-3},\tau_0\epsilon^{-3}]$ and $\epsilon \in (0,\epsilon_0)$.
\end{lem}

\begin{proof}
Coercivity (\ref{coercivity}) follows from the lower bound applied to (\ref{energy-type})
$$
2 \mathcal{E}(t) \geq \| \mathcal{Q}(t) \|_{\ell}^2 + \left(1
- \epsilon^2 p \| W(\cdot,\tau) \|_{L^{\infty}}^{p-1} \right) \| \mathcal{U}(t) \|_{\ell^2}^2.
$$
Since $1 - \epsilon^2 p \| W(\cdot,\tau) \|_{L^{\infty}}^{p-1} \geq 1/2$, we obtain the bound (\ref{coercivity}).
Note that if $p$ is odd, the lower bound for (\ref{energy-type}) implies (\ref{coercivity}) with $2 \mathcal{E}(t)$
replacing $4 \mathcal{E}(t)$.

Taking derivative of $\mathcal{E}$ with respect to time $t$ and using the perturbed FPU lattice equations (\ref{FPU-lattice-time})
yield the evolution of the energy-type quantity:
\begin{equation*}
\begin{split}
\frac{d \mathcal{E}}{d t} = \sum_{n \in \mathbb{Z}}
& \Big[ \mathcal{Q}_n(t) \mathcal{R}_n(W,\mathcal{U})(t) + \mathcal{Q}_n(t) {\rm Res}_n^{(2)}(t)
+ \mathcal{U}_{n}(t) \left[ 1 + \epsilon^2 p W^{p-1}(\epsilon(n-t),\epsilon^3t) \right]  {\rm Res}_n^{(1)}(t) \\
& + \frac{1}{2} \epsilon^2 p (p-1) W^{p-2}(\epsilon(n-t),\epsilon^3t) \mathcal{U}^2_{n}(t)
(-\epsilon \partial_{\xi} + \epsilon^3 \partial_{\tau}) W(\epsilon(n-t),\epsilon^3t) \Big].
\end{split}
\end{equation*}
By using the Cauchy--Schwarz inequality, we estimate
\begin{eqnarray*}
\left| \frac{d \mathcal{E}}{d t} \right|
& \leq & \| \mathcal{Q} \|_{\ell^2} \| \mathcal{R}(W,\mathcal{U}) \|_{\ell^2}
+ \| \mathcal{Q} \|_{\ell^2} \| {\rm Res}^{(2)} \|_{\ell^2} + \,\frac32\, \| \mathcal{U} \|_{\ell^2}
\left\| {\rm Res}^{(1)} \right\|_{l^2} \\
& \phantom{t} & + \frac{1}{2} \epsilon^3 p (p-1) \| W \|_{L^{\infty}}^{p-2} \left( \| \partial_{\xi} W(\cdot,\tau) \|_{L^{\infty}} +
\epsilon^2 \| \partial_{\tau} W(\cdot,\tau) \|_{L^{\infty}} \right) \| \mathcal{U}(t) \|_{\ell^2}^2.
\end{eqnarray*}
By using estimates (\ref{bound-rem}) and (\ref{bound-non}) in Lemma \ref{lemma-residual}, the generalized KdV equation (\ref{genKdV}) for $W$,
and the coercivity bound (\ref{coercivity}), we finally obtain the estimate (\ref{derivative-estimates}).
\end{proof}

\section{Proof of Theorem \ref{theorem-justification-1}}

Here we use the formalism of Section 2 and prove Theorem \ref{theorem-justification-1}.

We consider global solutions to the KdV and modified KdV equations ($p=2,3$). In this case, it is known \cite{Nguenn}
that for any global solution $W \in C(\mathbb{R},H^s(\mathbb{R}))$ with $s \geq 6$,
there exists a positive constant $\delta$ that only depends on the initial value of $W$ at $\tau = 0$
such that
\begin{equation}
\label{bound-norm}
\| W(\cdot,\tau) \|_{H^s} \leq \delta \quad \mbox{\rm for every} \;\; \tau \in \mathbb{R}.
\end{equation}
In what follows, we neglect mentioning that all constants depend on the choice of $s \geq 6$.

For any initial data $(u_{\rm in},\dot{u}_{\rm in}) \in \ell^2(\mathbb{R})$ satisfying the bound
(\ref{bound-initial-time-1}), there exists a local solution $(u,\dot{u}) \in C^1((-t_0,t_0),\ell^2(\mathbb{Z}))$
to the FPU equation (\ref{FPUlattice}). Equivalently, there exists a local solution
$(u,q) \in C^1((-t_0,t_0),\ell^2(\mathbb{Z}))$ to the FPU lattice equations \eqref{FPUvector}.
The solution can be decomposed according to equation \eqref{decomposition-time}.

Let us set $\mathcal{S} := \mathcal{E}^{1/2}$, where $\mathcal{E}(t)$ is the energy-type quantity
defined by (\ref{energy-type}).  The initial bound \eqref{bound-initial-time-1} ensures that
$\mathcal{S}(0) \leq C_0 \epsilon^{3/2}$ for some constant $C_0 > 0$, if $\epsilon_0$ is
chosen by (\ref{epsilon-0}). For fixed $\epsilon$-independent constants
$r \in \left(0,\frac{1}{2}\right)$, $C > C_0$, and $K > 0$,
let us define the maximal continuation time
\begin{equation}
\label{time-maximal}
T_{C,K,r} := \sup\left\{ T_0 \in (0,r K^{-1} \epsilon^{-3} |\log(\epsilon)|] : \quad
\mathcal{S}(t) \leq C \epsilon^{3/2-r}, \;\; t \in [-T_0,T_0] \right\}.
\end{equation}
Let us also define the maximal evolution time of the generalized KdV equation
(\ref{genKdV}) by $\tau_0(\epsilon) = r K^{-1} |\log(\epsilon)|$.
The energy estimate (\ref{derivative-estimates}) of Lemma \ref{lemma-energy}
can be rewritten for the variable $\mathcal{S}$ by
\begin{equation}
\label{estimate-1}
\left| \frac{d \mathcal{S}}{d t} \right| \leq
C_1 (\delta + \delta^{2p-1}) \epsilon^{9/2} + \epsilon^3 C_2 \left[ (\delta^{p-1} + \delta^{2p-2}) +
\epsilon^{-1} (\delta^{p-2} + \mathcal{S}^{p-2}) \mathcal{S} \right]  \mathcal{S},
\end{equation}
where $C_1$ and $C_2$ are positive constants that are independent of $\delta$ and $\epsilon$.
Since $\delta$ is independent of the maximal existence time $\tau_0(\epsilon)$ and hence is independent of $\epsilon$,
we can choose an $\epsilon$-independent positive constant $K$ sufficiently large such that
\begin{equation}
\label{estimate-2}
C_2 \left[ (\delta^{p-1} + \delta^{2p-2}) +
\epsilon^{-1} (\delta^{p-2} + C^{p-2} \epsilon^{(3/2-r)(p-2)}) C \epsilon^{3/2-r} \right] \leq K,
\end{equation}
as long as $\mathcal{S}(t) \leq C \epsilon^{3/2-r}$, since $r \in \left(0,\frac{1}{2}\right)$.
Using (\ref{estimate-1}) and (\ref{estimate-2}), we obtain
\begin{equation}
\label{estimate-3}
\left| \frac{d}{d t}  e^{-\epsilon^3 K t} \mathcal{S} \right| \leq
C_1 (\delta + \delta^{2p-1}) \epsilon^{9/2} e^{-\epsilon^3 K t}.
\end{equation}
By Gronwall's inequality, we obtain
\begin{equation}
\label{estimate-G}
\mathcal{S}(t) \leq
\left( \mathcal{S}(0) + K^{-1} C_1 (\delta + \delta^{2p-1}) \epsilon^{3/2} \right) \, e^{K \tau_0(\epsilon)},
\quad t \in [-T_{C,K,r},T_{C,K,r}],
\end{equation}
where the exponent is extended to the maximal existence time $t_0(\epsilon) := \epsilon^{-3} \tau_0(\epsilon)$.
From the definition of $\tau_0(\epsilon)$, we have
$$
\tau_0(\epsilon) = r K^{-1} |\log(\epsilon)| \quad \Rightarrow \quad e^{K \tau_0(\epsilon)} =  \epsilon^{-r}.
$$
Thus, we get from (\ref{estimate-G}),
\begin{equation} \label{lastGronwall}
\mathcal{S}(t) \leq
\left( C_0 + K^{-1} C_1 (\delta + \delta^{2p-1})  \right) \, \epsilon^{3/2-r}, \quad t \in [-T_{C,K,r},T_{C,K,r}].
\end{equation}
One can choose an $\epsilon$-independent constant $C > C_0$ sufficiently large such that
\begin{equation} \label{estimate-4}
C_0 + K^{-1} C_1 (\delta + \delta^{2p-1}) \leq C.
\end{equation}
Under the constraints (\ref{estimate-2}) and (\ref{estimate-4}) on $C > C_0$ and $K > 0$,
the time interval in (\ref{time-maximal}) can be extended to the maximal interval with
$T_{C,K,r} = t_0(\epsilon) = \epsilon^{-3} \tau_0(\epsilon)$.
Bound (\ref{bound-final-time-1}) of Theorem \ref{theorem-justification-1} is proved
due to the estimate (\ref{lastGronwall}) and the coercivity bound (\ref{coercivity}),
while the definition of $C$ may need a minor adjustment.

\section{Proof of Theorem \ref{theorem-justification-2}}

Here we use the formalism of Section 2 and prove Theorem \ref{theorem-justification-2}.

We consider global solutions to the generalized KdV equations with $p \geq 4$. It follows from \cite{Staf}
that for any global solution $W \in C(\mathbb{R},H^s(\mathbb{R}))$ with $s \geq 6$, there exists positive constants
$A$ and $K$ such that
\begin{equation}
\label{bound-norm-2}
\delta(\tau) := \sup_{\tau' \in [-\tau,\tau]} \| W(\cdot,\tau') \|_{H^s} \leq A (1 + |\tau|^{s-1}) \leq A e^{K |\tau|}
\quad \mbox{\rm for every} \;\; \tau \in \mathbb{R}.
\end{equation}
Again, we neglect mentioning that all constants depend on the choice of $s \geq 6$.
We note that the global solution to the generalized KdV equation (\ref{genKdV}) with $p \geq 5$ exists
if the $H^{s_p}(\mathbb{R})$ norm of the initial value $W$ at $\tau = 0$ is small, where $s_p$ is given by (\ref{s-p}) \cite{KPV2}.

For any initial data $(u_{\rm in},\dot{u}_{\rm in}) \in \ell^2(\mathbb{R})$ satisfying the bound
(\ref{bound-initial-time-2}), there exists a local solution $(u,\dot{u}) \in C^1((-t_0,t_0),\ell^2(\mathbb{Z}))$
to the FPU equation (\ref{FPUlattice}), or equivalently, a local solution
$(u,q) \in C^1((-t_0,t_0),\ell^2(\mathbb{Z}))$ to the FPU lattice equations \eqref{FPUvector}.
The solution can be decomposed according to equation \eqref{decomposition-time}.

Let us set $\mathcal{S} := \mathcal{E}^{1/2}$, where $\mathcal{E}(t)$ is the energy-type quantity
defined by (\ref{energy-type}).  The initial bound \eqref{bound-initial-time-1} ensures that
$\mathcal{S}(0) \leq C_0 \epsilon^{3/2}$ for some constant $C_0 > 0$, if $\epsilon_0$ is
chosen by (\ref{epsilon-0}). For fixed $\epsilon$-independent constants $r \in \left(0,\frac{1}{2}\right)$,
$C > C_0$, and $K > 0$, where $K$ is the same as in the bound (\ref{bound-norm-2}), let us define the maximal continuation time
\begin{equation}
\label{time-maximal-2}
T_{C,K,r} := \sup\left\{ T_0 \in \left( 0,(2 p K)^{-1} \epsilon^{-3} \log\left( r |\log(\epsilon)|\right)\right] : \quad
\mathcal{S}(t) \leq C \epsilon^{3/2-r}, \;\; t \in [-T_0,T_0] \right\}.
\end{equation}
Note that the maximal evolution time of the generalized KdV equation
(\ref{genKdV}) by
$$
\tau_0(\epsilon) = (2 p K)^{-1} \log\left( r |\log(\epsilon)|\right)
$$
is chosen differently compared to the proof of Theorem \ref{theorem-justification-1}.
The energy estimate (\ref{derivative-estimates}) of Lemma \ref{lemma-energy}
is rewritten for $\mathcal{S}$ by
\begin{equation}
\label{estimate-1-2}
\left| \frac{d \mathcal{S}}{d t} \right| \leq
C_1 (\delta(\tau) + \delta^{2p-1}(\tau)) \epsilon^{9/2} + \epsilon^3 C_2 \left[ (\delta^{p-1}(\tau) + \delta^{2p-2}(\tau)) +
\epsilon^{-1} (\delta^{p-2}(\tau) + \mathcal{S}^{p-2}) \mathcal{S} \right]  \mathcal{S},
\end{equation}
where $C_1$ and $C_2$ are positive constants that are independent of $\delta$ and $\epsilon$.
Since $\delta$ depends on time $\tau$ according to the bound (\ref{bound-norm-2}),
we may choose the $\epsilon$-independent positive constant $K$ in the bound (\ref{bound-norm-2})
sufficiently large such that
\begin{equation}
\label{estimate-2-2}
C_2 \left[ (\delta^{p-1}(\tau) + \delta^{2p-2}(\tau)) +
\epsilon^{-1} (\delta^{p-2}(\tau) + C^{p-2} \epsilon^{(3/2-r)(p-2)}) C \epsilon^{3/2-r} \right]  \leq 2 p K e^{2 p K |\tau|},
\end{equation}
as long as $\mathcal{S}(t) \leq C \epsilon^{3/2-r}$, since $r \in \left(0,\frac{1}{2}\right)$.
Using (\ref{estimate-1-2}) and (\ref{estimate-2-2}), we obtain
\begin{equation}
\label{estimate-3-2}
\left| \frac{d}{d t}  e^{-e^{2 \epsilon^3 p K t}} \mathcal{S} \right| \leq
B \epsilon^{9/2} e^{\epsilon^3 (2p-1) K t} e^{-e^{2 \epsilon^3 p K t}},
\end{equation}
where $B > 0$ is another $\epsilon$-independent constant and the inequality (\ref{estimate-3-2}) is set for $t > 0$
for simplicity. The estimate for $t < 0$ are similar. By Gronwall's inequality, we obtain
\begin{equation}
\label{estimate-G-2}
\mathcal{S}(t) \leq
\left( e^{-1} \mathcal{S}(0) + B F_K) \epsilon^{3/2} \right) \, e^{e^{2 p K \tau_0(\epsilon)}},
\quad t \in [-T_{C,K,r},T_{C,K,r}],
\end{equation}
where the exponent is extended to the maximal existence time $t_0(\epsilon) := \epsilon^{-3} \tau_0(\epsilon)$
and the positive $\epsilon$-independent constant $F_K$ is defined by
$$
F_K := \int_0^{\infty} e^{(2p-1) K \tau} e^{-e^{2 pK \tau}} d \tau < \infty.
$$
From the definition of $\tau_0(\epsilon)$, we have
$$
\tau_0(\epsilon) = (2 p K)^{-1} \log\left(r |\log(\epsilon)|\right) \quad \Rightarrow \quad
e^{e^{2 p K \tau_0(\epsilon)}} =  \epsilon^{-r}.
$$
Thus, we get from (\ref{estimate-G-2}),
\begin{equation}
\label{lastGronwall-2}
\mathcal{S}(t) \leq \left( C_0 + B F_K \right) \, \epsilon^{3/2-r}, \quad t \in [-T_{C,K,r},T_{C,K,r}].
\end{equation}
Note that $F_K \to 0$ as $K \to \infty$. One can choose an $\epsilon$-independent constant $C > C_0$
sufficiently large such that
\begin{equation}
\label{estimate-4-2}
C_0 + B F_K \leq C.
\end{equation}
Under the constraints (\ref{bound-norm-2}), (\ref{estimate-2-2}) and (\ref{estimate-4-2}) on $C > C_0$ and $K > 0$,
we can extend the time interval in (\ref{time-maximal-2}) to the maximal interval with
$T_{C,K,r} = t_0(\epsilon) = \epsilon^{-3} \tau_0(\epsilon)$.
Bound (\ref{bound-final-time-2}) of Theorem \ref{theorem-justification-2} is proved
due to the estimate (\ref{lastGronwall-2}) and the coercivity bound (\ref{coercivity}),
while the definition of $C$ may need a minor adjustment.


\begin{thebibliography}{99}

\bibitem{BP06} D. Bambusi and A. Ponno.
``On metastability in FPU", Comm. Math. Phys. {\bf 264} (2006), 539-561.

\bibitem{HW1}
G.N. Benes, A. Hoffman, and C.E. Wayne, ``Asymptotic stability of the Toda m-soliton",
J. Math. Anal. Appl.  {\bf 386}  (2012), 445--460.

\bibitem{DumaJames} B. Bidegaray--Fesquet, E. Dumas, and G. James, ``From
  Newton's cradle to the discrete p-Schr\"{o}dinger equation'', SIAM
  J. Math. Anal. {\bf 45} (2013), 3404--3430.


\bibitem{Nguenn} J. Bona, Y. Liu, and N.V. Ngueyn, ``Stability fo solitary waves in higher-order Sobolev spaces",
Comm Math. Sci. {\bf 2} (2004), 35--52.



\bibitem{DumaPel} E. Dumas and D.E. Pelinovsky, ``Justification of the log-KdV
  equation in granular chains: the case of precompression'', SIAM
  J. Math. Anal. {\bf 46} (2014), 4075--4103.

\bibitem{Pego} G. Friesecke and R.L. Pego, ``Solitary waves on FPU
lattices, Nonlinearity {\bf 12} (1999),  1601-1627; {\bf 15} (2002), 1343-1359;
{\bf 17} (2004), 207-227; {\bf 17} (2004),  229-251.

\bibitem{GMWZ} J. Gaison, S. Moskow, J.D. Wright, and Q. Zhang, ``Approximation of polyatomic FPU
lattices by KdV equations", Multiscale Model. Simul. {\bf 12} (2014), 953--995.

\bibitem{Gardner} R.M. Miura, C.S. gardner, and M.D. Kruskal, ``Korteweg--de Vries equations and generalization.
II. Existence of conservation laws and constants of motion", J. Math. Phys. {\bf 9} (1968), 1204--1209.

\bibitem{HW2}
A. Hoffman and C.E. Wayne, ``Asymptotic two-soliton solutions in the Fermi-Pasta-Ulam model",
J. Dynam. Differential Equations  {\bf 21}  (2009),  343--351.


\bibitem{Kato1} T. Kato, ``On the Korteweg-de Vries equation", Manuscript Math.
{\bf 28} (1979), 89--99.

\bibitem{Kato2} T. Kato, ``On the Cauchy problem for the (generalized)
Korteweg-de Vries equation", Stud. Appl. Math. {\bf 8} (1983), 93--128.

\bibitem{KPV1} C. Kenig, G. Ponce, and L. Vega, ``Well-posedness of the initial-value problem
for the Korteweg--De Vries equation", J. Americ. Math. Soc. {\bf 4} (1991), 323--347.

\bibitem{KPV2} C. Kenig, G. Ponce, and L. Vega, ``Well-posedness and scattering results for
the generalized Korteweg--De Vries equation via the contraction principle",
Comm. Pure Appl. Math. {\bf 46} (1993), 527--620.

\bibitem{Lannes} D. Lannes and J. Rauch, ``Validity of nonlinear geometric optics with times growing 
logarithmically", Proc. AMS {\bf 129} (2000), 1087--1096.

\bibitem{miz1}
T. Mizumachi, ``Asymptotic stability of lattice solitons in the energy space",
Commun. Math. Phys. {\bf 288} (2009), 125-144.

\bibitem{miz2}
T. Mizumachi, ``Asymptotic stability of $N$-solitary waves of the FPU lattices",
Archive for Rational Mechanics and Analysis {\bf 207} (2013), 393-457.

\bibitem{PW1} R.L. Pego and M.I. Weinstein, ``Eigenvalues, and
instabilities of solitary waves", Philos. Trans. Roy. Soc. London A
{\bf 340} (1992), 47--94.

\bibitem{PW2} R.L. Pego and M.I. Weinstein, ``Asymptotic
stability of solitary waves", Comm. Math. Phys. {\bf 164} (1994), 305--349.

\bibitem{PPP} D. Pelinovsky, T. Penati, and S. Paleari, ``Approximation of
small-amplitude weakly coupled oscillators with discrete nonlinear Schr\"{o}dinger equations",
Rev. Math. Phys. (2016), submitted.

\bibitem{SW00} G. Schneider and C.E. Wayne, ``Counter-propagating waves on
fluid surfaces and the continuum limit of the Fermi-Pasta-Ulam model", In
{\em International Conference on Differential Equations
(Berlin, 1999), vol. 1 (eds B Fiedler, K Gr\"oger, J Sprekels)}, pp. 390--404
(World Sci. Publishing, River Edge, NJ, USA, 2000).

\bibitem{Staf} G. Staffilani, ``On the growth of high Sobolev norms of solutions for KdV
and Schr\"{o}dinger equations", Duke Math. J. {\bf 86} (1997), 109--142.

\end{thebibliography}
\end{document}